\documentclass[11pt, a4paper]{article}
\usepackage{xcolor}

\usepackage{amsmath}
\usepackage{amsthm}
\usepackage{amssymb}
\usepackage{latexsym}
\usepackage{graphicx}
\usepackage{setspace}
\usepackage{float}
\usepackage{enumerate}
\usepackage{color}
\usepackage{cite}
\usepackage[normalem]{ulem}
\usepackage[T1]{fontenc}
\usepackage[utf8]{inputenc}

\usepackage{enumitem}
\usepackage{algorithmicx}
\usepackage{algpseudocode}
\usepackage[Algorithm,ruled]{algorithm}

\newtheorem{theorem}{Theorem}[section]
\newtheorem{lemma}[theorem]{Lemma}
\newtheorem{corollary}[theorem]{Corollary}
\newtheorem{definition}[theorem]{Definition}
\newtheorem{proposition}[theorem]{Proposition}

\theoremstyle{definition}
\newtheorem{example}[theorem]{Example}
\theoremstyle{remark}
\newtheorem{remark}[theorem]{Remark}

\newcommand{\RR}{\mathbb R}
\newcommand{\DD}{\mathcal D}

\newcommand{\VV}{\mathcal V}
\newcommand{\MM}{\mathcal M}

\newcommand{\tr}{{\rm tr}}

\newcommand{\bx}{\bar x}
\newcommand{\AutV}{\operatorname{Aut}(\VV)}
\newcommand{\DerV}{\operatorname{Der}(\VV)}

\newcommand{\ip}[2]{\left< #1,\, #2 \right>}

\newcommand{\Rn}{\mathbb{R}^n}

\usepackage{algorithmicx}
\usepackage{algpseudocode}

\usepackage{enumitem}
\usepackage{algcompatible}
\usepackage{enumerate}

\begin{document}

\title{Commutation principles for optimization problems involving strictly Schur-convex functions in Euclidean Jordan algebras}

\author{Pedro G. Massey\footnote{CMaLP, Facultad de Cs. Exactas, Universidad Nacional de La Plata e Instituto Argentino de Matemática, CONICET, Argentina. (e-mails: massey@mate.unlp.edu.ar and nbrios@mate.unlp.edu.ar). Partially supported by CONICET (PIP 11220200103209CO) and Universidad Nacional de La Plata (UNLP 11X974), Argentina.} \and Noelia B. Rios$^*$. 
\and David Sossa\footnote{Universidad de O'Higgins, Instituto de Ciencias de la Ingenier\'ia, Av.\,Libertador Bernardo O'Higgins 611, Rancagua, Chile (e-mail: david.sossa@uoh.cl). Partially supported by  FONDECYT (Chile)  through grant 11220268, and MATH-AMSUD 23-MATH-09 MORA-DataS project.}    } 

\date{}

\maketitle

\bigskip
\bigskip

\begin{quote}{\small 
    \textbf{Abstract}. In this work we establish several commutation principles 
		for optimizers of shifts of spectral functions in the context of Euclidean Jordan Algebras
		(EJAs). For instance, we show that under certain assumptions, if $\bx$ is a (local) optimizer of $F(x-a)$ for $x\in\Omega$, where $\Omega\subset \VV$ is a spectral set of an EJA $\VV$, $a\in \VV$ and
		$F:\VV\rightarrow \mathbb R$ is a strictly Schur-convex spectral function, then $a$ and $\bx$ operator commute. We make no further assumption on the smoothness of $F$; instead, we take advantage of the smoothness 
		(Lie structure) of the Automorphism group of $\VV$ and make use of majorization techniques for the eigenvalues of elements in EJAs. Our approach allows us to deal with several problems considered in the literature, related to strictly convex spectral functions and strictly convex spectral norms. In particular, we use our commutation principles to analyze the problem of minimizing the condition number in EJAs. 
    
    \bigskip
    
    
     }\end{quote}
\medskip



\section{Introduction}
Commutation principles in Euclidean Jordan algebras (EJAs) consist of a class of optimization problems involving (weakly) spectral sets and functions where it is ensured that each (local) optimizer operator commutes with some element related to the objective function. For instance, let $\mathcal V$ be a EJA,  consider the map
\begin{equation}\label{map0}
x\in\Omega\mapsto \Theta(x)+F(x),
\end{equation}
where $\Omega\subseteq\VV$, and $\Theta,\,F:\VV\to\mathbb R$. In \cite{RSS}, Ram\'irez, Seeger and Sossa have proved that when $\Omega$ is a spectral set, $\Theta$ is differentiable, and $F$ is a spectral function (see definitions in the next section), if $\bx$ is a local optimizer (minimizer or maximizer) of \eqref{map0}, then $\nabla\Theta(\bx)$ and $\bx$ operator commute. Gowda and Jeong \cite{GJ2017} further developed this commutation principle by assuming that $\Omega$ and $F$  are weakly spectral (i.e., they are invariant under automorphisms, which is a weaker condition than spectrality). These commutation principles were successfully applied in some variational problems like obtaining a meta-formula for generalized subdifferential of spectral functions \cite{Lourenco}, and deducing bounds for the determinant of the sum of two elements in EJAs \cite{Sossa}. It has also been investigated strong commutation principles for this class of problems, which consist of commutation principles that provide strong operator commutation. For instance, Gowda \cite{G2022} proved that under the assumption that $\Omega$ and $F$ are spectrally defined, and by setting $\Theta(x)=\langle x,a\rangle$, with $a\in\VV$,
\begin{equation}\label{strong}
\text{$\bx$ is a global maximizer of \eqref{map0} \quad $\Rightarrow$\quad $a$ and $\bx$ strongly operator commute.}
\end{equation}
This class of results allows us, for instance, to simplify the problem of minimizing \eqref{map0} to a minimization problem formulated in $\RR^n$, with $n$ being the rank of $\VV$, via the eigenvalue map.


In this paper, we study commutation principles for the problem of optimizing the map
\begin{equation}\label{map}
x\in\Omega\mapsto F(x-a),
\end{equation}
where $a\in\VV$, $\Omega\subseteq \VV$ and $F:\VV\to\mathbb R$ is a spectral function. Observe that, in general, the maps \eqref{map} and \eqref{map0} are not related; so, we can not use directly the commutation principle discussed above to characterize the optimizers of \eqref{map}. Indeed, despite the fact that $F$ is spectral, the shifted spectral function $F(\cdot-a)$ is no longer spectral. Moreover, we are not imposing a differentiability assumption over $F$.  Some particular cases of \eqref{map} have been studied in the literature. For instance, Jeong and Sossa \cite{JS2024} have recently proved that when $\Omega$ is weakly spectral and $F$ is a strictly convex spectral function or a strictly convex spectral norm, if $\bx$ is a local optimizer of \eqref{map} then $a$ and $\bx$ operator commute. Their proof relies on a commutation principle developed for \eqref{map0} by using nonsmooth analysis. Massey, Rios and Stojanoff \cite{Massey} have also obtained a commutation principle for \eqref{map}
in the context of Hermitian matrices, which are examples of (simple) EJAs. Their result states that in case $\Omega$ is an orbit of some element of $\VV$ and $F$ is a strictly convex spectral norm, if $\bx$ is a local minimizer of \eqref{map} then $a$ and $\bx$ strongly operator commute. Their approach is based on majorization techniques and on the smooth manifold structure of unitary orbits of Hermitian matrices. Similar results have also been applied in some other optimization problems \cite{Benac}.

Inspired by the works of Massey, Rios and Stojanoff \cite{Massey}, and Jeong and Sossa \cite{JS2024}, our aim is to develop commutation principles and strong commutation principles for the problem of optimizing \eqref{map} with $\Omega$ and $F$ as general as possible. Our proof techniques are mainly inspired by those techniques developed in \cite{Massey}; that is, we will use majorization techniques in EJAs, see \cite{JJL2020}. The spectral function $F$ will be assumed to be strictly Schur-convex. Observe that this class of functions is quite broad and covers classes like spectral functions that are strictly convex, strictly quasi-convex, and strictly convex spectral norms. We also introduce a new strictly Schur-convex function which is used to deal with the problem of minimizing the condition number (see Section\,\ref{application}).

Our main results are the following commutation principles: Let $a,\,b\in\VV$, $\Omega\subseteq \VV$ and $F:\VV\to\mathbb R$ be a strictly Schur-convex spectral function;
\begin{enumerate}
\item[(i)] Suppose that $\Omega$ is a spectral set. If $\bx$ is a global minimizer (resp. global maximizer) of \eqref{map}, then $a$ and $\bx$ (resp. $-a$ and $\bx$) strongly operator commute (see Theorem\,\ref{th:strongly}).
\item[(ii)] Suppose that $\Omega$ is a weakly spectral set. If $\bx$ is a local optimizer (minimizer or maximizer) of \eqref{map}, then $a$ and $\bx$ operator commute (see Theorem\,\ref{teo local Lidskii in EJAs tutti}).
\item[(iii)] Suppose that $\VV$ is simple, and set $\Omega$ as the eigenvalue orbit of $b$. If $\bx$ is a local minimizer (resp. local maximizer) of \eqref{map}, then $a$ and $\bx$ (resp. $-a$ and $\bx$) strongly operator commute. Moreover, $\bx$ is a global optimizer of \eqref{map} (see Theorem\,\ref{teo local Lidskii in EJAs}).
\end{enumerate}

 Regarding item ({\rm i}) above, we remark that in this case, if we replace the hypothesis that $\bx$ is a global minimizer (resp. global maximizer) 
by the weaker hypothesis that $\bx$ is a local minimizer (resp. local maximizer) 
of \eqref{map} 
then $a$ and $\bx$ (resp. $-a$ and $\bx$) do not necessarily strongly operator commute 
(see Example \ref{rem no ejem2}). We also point out that if we replace the assumption that $\Omega$ is a spectral set for the (weaker) assumption that 
$\Omega$ is a weakly spectral set and assume that $\bx$ is a global minimizer (resp. global maximizer) of \eqref{map} 
then, $a$ and $\bx$ (resp. $-a$ and $\bx$) do not necessarily strongly operator commute (see Example \ref{no exa1} below).

 Regarding item ({\rm ii}), notice that weakly spectral sets are rather small domains (as opposed to spectral sets, that are larger domains) for an spectral function in an arbitrary EJA. Hence, the hypothesis that $\bx$ is a local optimizer within weakly spectral sets is a rather weak assumption. In particular, the result also applies to local optimizers within spectral sets (see Corollary \ref{coro local Lidskii in EJAs x28}).

 Regarding item ({\rm iii}), we point out that if we remove the hypothesis that $\VV$ is simple and assume that $\bx$ is a local minimizer (resp. local maximizer) of \eqref{map} then, $a$ and $\bx$ (resp. $-a$ and $\bx$) do not necessarily strongly operator commute (see Example \ref{no exa1} below). Also, notice that we have set $\Omega$ to be a rather minimal set, and hence, the hypothesis that $\bx$ is a local minimizer in $\Omega$ is quite weak.

 As a consequence of the previous comments, notice that the results in items $({\rm i})$-$({\rm iii})$ are sharp.

The organization of the paper is as follows: In Section\,\ref{preli}, we provide a background of the concepts and tools from EJAs that we require for developing our results. In Section\,\ref{majorization}, we review majorization concepts and provide examples of strictly Schur-convex spectral functions. Sections \ref{strong principle}, \ref{general comm}, and \ref{strong local} are mainly devoted to proving the respective statements (given above) (i), (ii), and (iii). In Section\,\ref{application}, we illustrate the applications of our commutation principles to the analysis of the problem of minimizing the condition number in EJA.

\section{Preliminaries}\label{preli}
\subsection*{Euclidean Jordan algebras}
Let $(\VV,\langle\cdot,\cdot\rangle,\circ)$ be \emph{Euclidean Jordan algebra} (EJA) of rank $n$ and unit element $e$. For any $x,y\in \VV$, $\langle x,y\rangle$ denotes the inner product and  $x\circ y$ denotes the Jordan product. We refer the reader to the book of Faraut and Kor\'anyi \cite{FK} as a general reference for the theory of EJAs.

Here, we recall some definitions and properties of EJAs that will be used later. 
The Jordan product in $\VV$ satisfies the commutation law ($x\circ y=y \circ x$, for $x,\,y\in\VV$), 
but it is not associative in general; still, the Jordan product satisfies the Jordan identity: $(x\circ y)\circ (x\circ x)= x\circ (y\circ (x\circ x))$, for $x,\,y\in\VV$. 
Recall that an EJA is said to be \emph{simple} if it is not a direct product of nonzero EJAs. It is well known \cite[Corollary IV.1.5, Theorem V.3.7]{FK} that any EJA can be decomposed into a direct product of simple EJAs and there are only five simple ones up to isomorphism. One of them is the algebra $\mathcal S^n$ of $n\times n$ symmetric matrices with Jordan and inner products given by
\[ X \circ Y = \frac{XY + YX}{2}, \quad \ip{X}{Y} = \operatorname{tr}(XY)\;\mbox{(trace inner product)}. \]
An element $c\in \mathcal {V}$ is an \emph{idempotent} if $c^2 = c$. An
idempotent $c$ is  \emph{primitive} if it is nonzero and cannot be written as a sum of two nonzero idempotents. A  \emph{Jordan frame} is a collection $\{c_1,\ldots,c_n\}$ of primitive idempotents satisfying
$\sum_{i=1}^n c_i=e$, and $c_i\circ c_j=0$ when $i \neq j.$

The spectral decomposition theorem asserts that any element $x \in \VV$ can be written as
\[ x = \lambda_1(x) e_1 + \lambda_2(x) e_2 + \cdots + \lambda_n(x) e_n, \]
where $\lambda_1(x) \geq \lambda_2(x) \geq \cdots \geq \lambda_n(x)$ are real numbers called the eigenvalues of $x$, and $\{e_1, e_2, \ldots, e_n\}$ is a Jordan frame. We denote $\lambda(x)\in\mathbb R^n$ (eigenvalue map) to the vector whose components are the eigenvalues of $x$ arranged in nonincreasing order. The \emph{trace} of $x\in\VV$ is defined as the sum of its eigenvalues. That is, $\tr(x):=\lambda_1(x)+\cdots+\lambda_n(x).$ The inner product of $\VV$ is scaled so that we can assume that $\langle x,y\rangle=\tr(x\circ y)$ which is the \emph{trace inner product}. From now on, $\|\cdot\|$ will denote the norm on $\VV$ induced by the trace inner product; that is, $\|x\|=(\tr(x\circ x))^{1/2}$. When $\VV=\RR^n$, $\|\cdot\|$ coincides with the Euclidean norm (2-norm) on $\RR^n$.

Let $p$ be an idempotent element of $\VV$. It is known that $\VV$ can be decomposed as the direct sum
\begin{equation}\label{peirce}
\VV = \VV(p,1)\oplus \VV(p,0) \oplus \VV(p,1/2),
\end{equation}
where $\VV(p,\alpha):=\{x\in\VV:x\circ p=\alpha x\}$, for $\alpha\in\{1,0,1/2\}$.  This decomposition is called the \emph{Peirce decomposition} of $\VV$ with respect to $p$. It is also known that $\VV(p,1)$ and $\VV(p,0)$ are subalgebras of $\VV$, and $\VV(p,1/2)\neq\{0\}$ whenever $\VV$ is simple and $p\notin\{0,e\}$.

Two elements $a$ and $b$ in $\VV$ are said to \emph{operator commute} if $L_a L_b=L_b L_a$, where $L_a:\VV\to\VV$ is a linear map defined as $L_a(x)=a\circ x$ for all $x\in\VV$. The operator commutativity property in $\mathcal S^n$ coincides with the usual matrix commutativity property. That is, $A$ and $B$ operator commute in $\mathcal S^n$ if and only if $AB=BA$. It is known that $a, b \in \VV$ operator commute if and only if both $a$ and $b$ can be \emph{simultaneously decomposed} by the same Jordan frame. That is, there exist a Jordan frame $\{e_1, e_2, \ldots, e_n\}$ and real numbers $\alpha_i, \beta_i \in \mathbb{R}$ for $i = 1, 2, \ldots, n$ such that
\begin{equation}\label{eq op commute1} 
 a = \alpha_1 e_1 + \alpha_2 e_2 + \cdots + \alpha_n e_n, \quad b = \beta_1 e_1 + \beta_2 e_2 + \cdots + \beta_n e_n, 
\end{equation}
where $\alpha$ and $\beta$ are some rearrangements of $\lambda(a)$ and $\lambda(b)$, respectively. We say that $a, b \in \VV$ \emph{strongly operator commute}
 if there exists a frame $\{e_1, e_2, \ldots, e_n\}$ for $\VV$ such that the representations
in Eq. \eqref{eq op commute1} hold, in such a way that $\alpha = \lambda(a)$ and $\beta = \lambda(b)$.
Then, one can see that $a$ and $b$ strongly operator commute if and only if $\ip{a}{b} = \ip{\lambda(a)}{\lambda(b)}$.

\subsection*{Orbits, spectral sets, and spectral functions}

A linear map $X : \VV \to \VV$ is an \emph{(Jordan algebra) automorphism} if $X$ is bijective (invertible) and $X(a \circ b) = Xa \circ Xb,$ for all $a, b \in \VV$. The (Lie) group of all automorphisms of $\VV$ is denoted by $\AutV$. As an example, the automorphism group ${\rm Aut}(\mathcal S^n)$ of the EJA of $n\times n$ symmetric matrices can be identified with $\mathcal O(n)$ i.e. the orthogonal group in dimension $n$.

For $b\in\VV$, there are two kinds of orbits of $b$. The first one is the \emph{eigenvalue orbit} which consists of the elements of $\VV$ with the same eigenvalues of $b$; that is
$$[b]:=\{x\in\VV:\lambda(x)=\lambda(b)\}.$$
The second one is the \emph{Automorphism orbit} (or weak orbit), which is defined as
$$[b]_w:=\{Xb:X\in\AutV\}.$$
We have that $[b]_w\subseteq [b]$ and the equality holds whenever $\VV$ is simple.

The above orbits are the basic form of sets known as spectral and weakly spectral sets, respectively. We say that a set $Q\subseteq \mathbb R^n$ is \emph{symmetric} if it is permutation invariant which means that if $u\in Q$ then $Pu\in Q$ for every permutation matrix $P$ of order $n$. A set $\Omega\subset\VV$ is called \emph{spectral set} if there exists a symmetric set $Q\subseteq\mathbb R^n$ such that $\Omega=\{x\in\VV:\lambda(x)\in Q\}$. $\Omega$ is called \emph{weakly spectral set} if it is invariant under automorphisms which means that if $x\in\Omega$ and $A\in\AutV$ then $Ax\in\Omega$.  In particular, $[b]$ is a spectral set and $[b]_w$ is a weakly spectral set. It is known that any spectral set is weakly spectral, because of $\lambda(Ax)=\lambda(x)$ for all $A\in\AutV$ and $x\in\VV$. Moreover, the concepts of spectral and weakly spectral coincide whenever $\VV$ is simple.

We say that a function $f:\mathbb R^n\to\mathbb R$ is \emph{symmetric} if it is permutation invariant which means that if $u\in\mathbb R^n$ then $f(u)=f(Pu)$ for every permutation matrix $P$ of order $n$.  A function $F:\VV\to\mathbb R$ is said to be \emph{spectral function} if there exists a symmetric function $f:\mathbb R^n\to\mathbb R$ such that $F(x)=f(\lambda(x))$ for all $x\in\VV$.

\section{Majorization and strictly Schur-convex functions}\label{majorization}
For $u \in \Rn$, let $u^{\downarrow}\in \Rn$ denote the vector obtained from $u$ by arranging its entries in non-increasing order. We denote $(\RR^n)^{\downarrow}:=\{u\in\RR^n:u=u^\downarrow\}$. For $u,v\in\Rn$, one says that $u$ is \emph{majorized} by $v$, in symbol $u\prec v$, if 
\begin{equation}\label{eq defi mayo1}
u^{\downarrow}_1+\cdots+ u^{\downarrow}_k\leq v^{\downarrow}_1+\cdots+ v^{\downarrow}_k \quad \text{ for every } \quad k=1,\ldots,n\, ,
\end{equation} 
and $u^{\downarrow}_1+\cdots+ u^{\downarrow}_n=v^{\downarrow}_1+\cdots+ v^{\downarrow}_n$. If only the conditions in \eqref{eq defi mayo1} hold, we say that $u$ is submajorized by $v$, and write $u\prec_w v$. 
We say that $u$ is \emph{strictly majorized} by $v$ if $u\prec v$ and $u^\downarrow\neq v^\downarrow$ ($u$ is not a rearrangement of $v$). For $x,y\in\VV$, we say that $x$ is majorized by $y$ in $\VV$ and write $x\prec y$ if $\lambda(x)\prec\lambda(y)$ in $\RR^n$. We say that $x$ is strictly majorized by $y$ in $\VV$ if $\lambda(x)$ is strictly majorized by $\lambda(y)$ in $\RR^n$.



A function $f:\mathbb R^n\to\mathbb R$ is said to be \emph{Schur-convex} if for all $u,v\in\mathbb R^n$,
$$u\prec v\;\Rightarrow\;f(u)\leq f(v).$$
If, in addition, $f(u)<f(v)$ whenever $u$ is strictly majorized by $v$, then $f$ is said to be \emph{strictly Schur-convex} (cf. \cite{Marshall}). 

\begin{definition}
Let $F:\VV\to\RR$ be a spectral function given by $F(x)=f(\lambda(x))$. We say that $F$ is a \emph{Schur-convex} (resp. strictly Schur-convex) spectral function if its associated symmetric function $f$ is Schur-convex (resp. strictly Schur-convex).
\end{definition}
Since the commutation principles that we developed in this work require considering strictly Schur-convex spectral functions, we list some examples of them below. Other examples can be found in \cite[Chapter\,3]{Marshall} and \cite{Stepniak}.
\begin{example}
Any strictly convex spectral function $F$ is strictly Schur-convex. Indeed, if $f$ is the associated symmetric function of $F$, in \cite{JS2024} was proved that $F$ is strictly convex if and only if so is $f$. Then, the result follows by recalling that any strictly convex symmetric function is strictly Schur-convex.
\end{example}
\begin{example}
We say that $F:\VV\to\RR$ is a \emph{strictly convex spectral norm} if it is a spectral function and a strictly convex norm in $\VV$. Recall that $F(\cdot)$ is a strictly convex norm in $\VV$ if it is a norm in $\VV$ satisfying $F(x+y)<2$ for any distinct nonzero $x,y\in\VV$ such that $F(x)=F(y)$. For instance, the Schatten $p$-norm of $x\in\VV$ given by $||x||_p:=(|\lambda_1(x)|^p+\cdots+|\lambda_n(x)|^p)^{1/p}$ is a strictly convex norm for $p>1$. In \cite{JS2024}, it is proved that any strictly convex spectral norm $F$ in $\VV$ is strictly Schur-convex. Indeed, $F$ is a strictly convex spectral norm in $\VV$ if and only if its associated symmetric function $f$ is a strictly convex symmetric norm in $\RR^n$. Moreover, any strictly convex symmetric norm in $\RR^n$ is strictly Schur-convex.
\end{example}
\begin{example}
Recall that $f:\RR^n\to\RR$ is a \emph{strictly quasi-convex} function if $f(\alpha u+(1-\alpha)v)<\max\{f(u),f(v)\}$ for all $u\neq v$ and $\alpha\in(0,1)$, see \cite[Definition\,2.2.2]{Cambini}. By using induction, one can prove that $f$ is strictly quasi-convex if and only if for all integer $k\geq 2$, and for all distinct vectors $u_1,\ldots,u_k$ in $\mathbb R^n$, $f(\sum_{i=1}^k\alpha_i u_i)<\max_{i=1,\ldots,k}f(u_i)$ for all $\alpha_i\in(0,1)$ such that $\sum_{i=1}^k\alpha_i=1$.

Let  $f:\RR^n\to\RR$ be a strictly quasi-convex symmetric function. We claim that $f$ is strictly Schur-convex. Indeed,
let $u,v\in\RR^n$ be such that $u\prec v$ and $u^\downarrow\neq v^\downarrow$. Since $f$ is strictly quasi-convex, then it is quasi-convex. It is known that a quasi-convex symmetric function is Schur-convex \cite{Marshall}. Then, $f(u)\leq f(v)$. We will prove that $f(u)<f(v)$ when $f$ is strictly quasi-convex. In fact, suppose that $f(u)=f(v)$. Since $u$ is strictly majorized by $v$, then $u$ is in the convex hull of $\{Pv:P\text{ is a $n\times n$ permutation matrix}\}$ but it is not an extreme point of that set. Hence, $u=\sum_{i=1}^k\alpha_i P_i v$, where $k\geq 2, \alpha_i>0$ for all $i$, $\sum_{i=1}^k\alpha_i=1$, $P_i$ is a permutation matrix for all $i$, and $P_iv\neq P_jv$ whenever $i\neq j$. Thus, we arrive to the following contradiction:
$$f(v)=f\left(\sum_{i=1}^k\alpha_i P_i v\right)<\max_{i=1,\ldots,k}f(P_iv)=f(v),$$
where the strict inequality is because of the strict quasi-convexity of $f$, and the right equality is because $f$ is symmetric. The claim follows.

Therefore, a spectral function is strictly Schur-convex whenever its associated symmetric function is strictly quasi-convex. 
\end{example}

\begin{example}\label{example ssc functions1}
 Let $\VV$ be an EJA of rank $n$ and let $\VV_+=\{b\in \VV\ : \ \lambda_n(b)>0\}$. 
Given $b\in \VV_+$, we can consider the condition number $\mathfrak{c}(b)=\lambda_1 (b)/ \lambda_n(b)>0$, which is a measure of the spread of the eigenvalues of $b$ (see \cite{Seeger2022} for some optimization problems related with the condition number in EJAs). Similarly, if 
$c\in \VV$
 we can consider the spectral spread of $c$, $\text{Sprd}(c)=\lambda_1(c)-\lambda_n(c)\geq 0$, which is also a measure of the spread of the eigenvalues of $c$. For the particular case when $\mathcal V=\mathcal S^n$ is the EJA of real symmetric $n\times n$ matrices, $\mathfrak{c}(b)$ and $\text{Sprd}(c)$ play a role in determining the stability properties of $b\in \mathcal S^n_+$ and $c\in \mathcal S^n$. It is well known that $\mathfrak{c}(\cdot):\VV_+\rightarrow \mathbb R$ and $\text{Sprd}(\cdot):\VV\rightarrow \mathbb R$ are Schur-convex functions (since $u\prec v$ implies that 
$\lambda_n(v)\leq \lambda_n(u)$ and $\lambda_1(v)\geq \lambda_1(u)$) but not strictly Schur-convex. In order to get the corresponding strictly Schur-convex function associated with 
$\mathfrak{c}(\cdot)$ and $\text{Sprd}(\cdot)$ we consider the functions
$$
\kappa_{\|\cdot\|}(b)=\| (\lambda_i(b)/\lambda_{n-i+1}(b))_{i=1}^{\lfloor n/2\rfloor}\|\quad \text{and}\quad \text{Sprd}_{\|\cdot\|}(c)=\|(\lambda_i(c)-\lambda_{n-i+1}(c))_{i=1}^{\lfloor n/2\rfloor}\|\,,
$$
where $b\in\VV_+$, $c\in\VV$, and $\|\cdot \|$ is the Euclidean norm in $\mathbb R^{ \lfloor n/2\rfloor}$ ($\lfloor n/2\rfloor$ denotes the integer part of $n/2$). 
The case $\kappa_{\|\cdot\|}$ is analyzed in Section\,\ref{application}. The analysis for the case  $\text{Sprd}_{\|\cdot\|}$ is analogous. See \cite{KA09} for the introduction of the vector valued spread of selfadjoint matrices and \cite{MSZ21,MSZ21b} for some of its basic properties and applications in $\mathcal S^n$.
\end{example}

One of the main tools used to develop our analysis is the \emph{Lidskii's inequality}. As mentioned in \cite{Lewis1999}, the Lidskii's inequality is one of the central tools for studying perturbation theory for the eigenvalues of symmetric matrices. For the general case of Euclidean Jordan algebras, the proof of the Lidskii's inequality can be found in \cite[Theorem 5.1.]{JJL2020}. The Lidskii's inequality says that for $a$ and $b$ in $\VV$, the following majorization is satisfied:
\begin{equation}\label{lidskii}
    \lambda(a)-\lambda(b)\prec \lambda(a-b).
    \end{equation}

We will also use the \emph{Ky Fan's inequality} which says that for $a$ and $b$ in $\VV$:
\begin{equation}\label{fan}
    \lambda(a+b)\prec \lambda(a)+\lambda(b).
    \end{equation}
    It is not difficult to see that the Ky Fan's inequality can be obtained from the Lidskii's inequality.

Below, we provide a characterization for the strong operator commutation property by using the Lidskii's inequality.

\begin{proposition}\label{lidskiistrong}
    Let  $a, b\in\mathcal V$. The following statements are equivalent:
    \begin{enumerate}
\item[(a)] $a$ and $b$ strongly operator commute.
\item[(b)]  $\displaystyle \lambda(a+b)= \lambda(a)+\lambda(b).$
\item[(c)] $\displaystyle (\lambda(a)-\lambda(b))^{\downarrow}=\lambda(a-b).$
\end{enumerate}
\end{proposition}
\begin{proof}
The equivalence $(a)\Leftrightarrow (b)$ is proved in \cite[Proposition\,2.6]{Gowda2019}. Let us prove $(a)\Leftrightarrow (c)$. Suppose that $a$ and $b$ strongly operator commute. Then there exist a Jordan frame $\{c_1,\ldots,c_n\}$ such that $a=\sum_{i=1}^n\lambda_i(a) c_i$ and $b=\sum_{i=1}^n\lambda_i(b)c_i$. Hence,
$a-b=\sum_{i=1}^n(\lambda_i(a)-\lambda_i(b))c_i$ which implies the desired result $\lambda(a-b)=(\lambda(a)-\lambda(b))^\downarrow$. Suppose now that $(c)$ holds. Let $x\in[a]$. Then, by observing that $\lambda(x)=\lambda(a)$, and by using the Lidskii's inequality we obtain
 \begin{equation}\label{eqlid1}
     \lambda(a-b)=(\lambda(x)-\lambda(b))^{\downarrow}\prec\lambda(x-b).
 \end{equation}
 Since the 2-norm function $\Vert\cdot\Vert$ in $\mathbb R^n$ is convex and symmetric, we have that it is Schur-convex. Then, from \eqref{eqlid1} we obtain $\Vert \lambda(a-b)\Vert\leq \Vert \lambda(x-b)\Vert$. Hence,
 $\Vert a-b\Vert\leq \Vert x-b\Vert$ for all $x\in[a]$
where the norm on $\VV$ is the one induced by the trace inner product. Hence, $a$ is global minimizer of $\frac{1}{2}\Vert x-b\Vert^2$ for $x\in[a]$. Note that this is equivalent to say that $a$ is a global maximizer of $\langle x,b\rangle -\frac{1}{2}\Vert x\Vert^2$ for $x\in[a]$. Thus, since $-\frac{1}{2}\Vert \cdot\Vert^2$ is a spectral function, we can use the strong commutation principle \eqref{strong} to deduce that $a$ and $b$ strongly operator commute. 
\end{proof}

\section{On global optimizers and strong commutation principles}\label{strong principle}

\begin{theorem}\label{th:strongly}
    Let $F:\mathcal V\to\mathbb R$ be a strictly Schur-convex spectral function with $f:\RR^n\to\RR$ as its associated symmetric function, let $\Omega\subseteq\VV$ be a spectral set with $Q\subseteq\RR^n$ as its associated symmetric set, and let $a\in \VV$. Then, the following statements are satisfied:
    \begin{enumerate}
    \item[(a)] If $\bx$ is a global minimizer of $F(x-a)$ for $x\in\Omega$, then $a$ and $\bx$ strongly operator commute. Moreover,
    \begin{equation}\label{igual}
    \min_{x\in\Omega} F(x-a)=\min_{u\in Q}f(u-\lambda(a))=f(\lambda(\bx)-\lambda(a)).
    \end{equation}
    \item[(b)] If $\bx$ is a global maximizer of $F(x-a)$ for $x\in\Omega$, then $-a$ and $\bx$ strongly operator commute. Moreover,
    \begin{equation}\label{igual2}
    \max_{x\in\Omega} F(x-a)=\max_{u\in Q}f(u+\lambda(-a))=f(\lambda(\bx)+\lambda(-a)).
    \end{equation}
    \end{enumerate}
\end{theorem}
\begin{proof}
    $(a)$. Let $\bx\in\Omega$ be a global minimizer of $F(x-a)$, for $x\in\Omega$. Suppose that $a$ has the spectral decomposition $a=\sum_{i=1}^n\lambda_i(a)c_i$ for some Jordan frame $\{c_1,\ldots,c_n\}$. Let us take $\hat x:=\sum_{i=1}^n\lambda_i(\bx)c_i$. Observe that $\lambda(\hat x)=\lambda(\bx)$ by construction. Hence, $\hat x\in [\bx]\subseteq \Omega$ since this last set is spectral. Then, $\hat x-a=\sum_{i=1}^n(\lambda_i(\bx)-\lambda_i(a))c_i$ which implies
    \begin{equation}\label{lidskii2}
    \lambda(\hat x-a)=(\lambda(\bar x)-\lambda(a))^\downarrow\prec\lambda(\bx-a),
    \end{equation}
    where the majorization is because of the Lidskii's inequality \eqref{lidskii}.
    Let $f$ be the symmetric function associated with $F$. Since $F$ is a strictly Schur-convex spectral function, then $f$ is a strictly Schur-convex symmetric function.  Then, from \eqref{lidskii2}, we obtain
    \begin{equation}\label{ineq}
    F(\hat x-a)=f(\lambda(\hat x-a))\leq f(\lambda(\bx-a))=F(\bx-a).
    \end{equation}
    Recall that $\bx$ is a global minimizer of $F(x-a)$ for $x\in\Omega$. Then, this fact and \eqref{ineq} imply
    $f(\lambda(\hat x-a))=f(\lambda(\bx-a))$. Hence, due to the strict Schur-convexity of $f$ we deduce that 
\begin{equation}\label{trieq}
\lambda(\hat x-a)=\lambda(\bx-a).
\end{equation}
Indeed, otherwise, $\lambda(\hat x-a)$ is strictly majorized by $\lambda(\bx-a)$ and the inequality \eqref{ineq} should be strict. Now, by combining the equality in \eqref{lidskii2} with \eqref{trieq} we deduce that
$\lambda(\bx-a)=(\lambda(\bar x)-\lambda(a))^\downarrow$. Thus, from Proposition\,\ref{lidskiistrong}, we conclude that $a$ and $\bx$ strongly operator commute.

To prove \eqref{igual}, observe that since $a$ and $\bx$ strongly operator commute, there exists a Jordan frame $\{e_1,\ldots,e_n\}$ such that $a=\sum_{i=1}^n\lambda_i(a)e_i$ and $\bx=\sum_{i=1}^n\lambda_i(\bx)e_i$. Then,
\begin{equation}\label{relation}
\min_{x\in\Omega}F(x-a)=F(\bx -a)=f(\lambda(\bx)-\lambda(a))\geq \min_{u\in Q}f(u-\lambda(a)).
\end{equation}
We claim that the equality is attained in \eqref{relation}. Indeed, let $\hat u$ be a global solution of the right-hand side problem of \eqref{relation}. If  $f(\hat u-\lambda(a))<f(\lambda(\bx)-\lambda(a))$, then we can take $\hat x=\sum_{i=1}^n\hat u_i e_i$, which is in $\Omega$ since $\lambda(\hat x)=\hat u^{\downarrow}\in Q$, obtaining
$$F(\hat x-a)=f(\hat u-\lambda(a))<f(\lambda(\bx)-\lambda(a))=F(\bx-a),$$
contradicting that $\bx$ is a global solution of the left-hand side problem of \eqref{relation}.

\noindent $(b)$. Let $\bx\in\Omega$ be a global maximizer of $F(x-a)$ for $x\in\Omega$. We consider the representation 
of $a$ as $a=\sum_{i=1}^n-\lambda_i(-a)c_i$, and we set $\hat x:=\sum_{i=1}^n\lambda_i(\bx)c_i$. Then, from the Ky Fan inequality \eqref{fan} we get
$$\lambda(\bar{x}-a)\prec \lambda(\bar{x})+\lambda(-a)=\lambda(\hat{x}-a).$$
Thus, $F(\bx-a)\leq F(\hat x-a)$, and since $\bx$ is a global maximizer we have $F(\bx-a)= F(\hat x-a)$. Therefore, by reasoning as in item $(a)$, we deduce that   $\lambda(\bar{x}-a)=\lambda(\bar{x})+\lambda(-a)$, which implies that $-a$ and $\bx$ strongly operator commute because of Proposition\,\ref{lidskiistrong}. The proof of \eqref{igual2} is analogous to the proof of \eqref{igual} with the difference that in this case we have $a=\sum_{i=1}^n\lambda_i(-a)e_i$ and $\bx=\sum_{i=1}^n\lambda_i(\bx)e_i$, for some Jordan frame $\{e_1,\ldots,e_n\}$, because $-a$ and $\bx$ strongly operator commute.

\end{proof}

It is natural to consider whether the above Theorem is still valid if $\Omega$ is assumed to be weakly spectral. The following example shows that this is not the case.

\begin{example}\label{no exa1} 
 Assume that $\mathcal S$ is a simple EJA of rank $2$, and consider the nonsimple EJA given by $\mathcal V= \mathcal S \times \mathcal S $. For $b\in\VV$, recall that $[b]_w$ is weakly spectral (but not necessarily spectral since the algebra is not simple). Below we construct $a,\,b,\, \bx\in \mathcal V$ 
such that $\bx\in [b]_w$ is a global minimizer of $F(x-a)$ for $x\in [b]_w$ (for any continuous function $F:\mathcal V\rightarrow \mathbb R$) and such that $a$ and $\bx$ do not strongly operator commute. 

Recall from \cite[Theorem\,5.1]{GJ2017} that in this case, each $\phi\in {\rm Aut}(\mathcal V)$ has the form
$\phi=( \phi_1,\phi_2)\circ \sigma
$, with $\phi_j\in {\rm Aut}(\mathcal S)$, $j=1,\,2$ and where $\sigma$ denotes the identity of $\mathcal V$ or the transposition
$\sigma((b_1,b_2))=(b_2,b_1)$, for $(b_1,b_2)\in\mathcal V$.
As a consequence, we conclude that given $b=(b_1,b_2)\in\mathcal V$
$$
[b]_w=\{\, x=(x_1,x_2)\in\mathcal V:\ (\lambda(x_1),\lambda(x_2))= (\lambda(b_1),\lambda(b_2)) \ \text{ or } \  (\lambda(x_1),\lambda(x_2))= (\lambda(b_2),\lambda(b_1)) \, \}\,,
$$ where $\lambda(x_j)\in\mathbb R^2$ denotes the eigenvalue vector of $x_j\in\mathcal S$, for $j=1,2$. 
Hence, the structure of $\mathcal V$ imposes certain restrictions on $a=(a_1,a_2)\in\mathcal V$ so that this element can be an element of the orbit $[b]_w$, beyond merely spectral conditions on $\lambda(a)\in\mathbb R^4$. Indeed, consider $b=(b_1,b_2)\in\mathcal V$ such that $\lambda(b_1)=(4,1)$ and $\lambda(b_2)=(3,2)$. Let $\{e_1,\,e_2\}$ be a Jordan frame for $\mathcal S$ and consider 
$a_1=4\,e_1+3\,e_2$ and $a_2=2\,e_1+1\,e_2$. Hence, 
$\lambda(a_1)=(4,3)$ and $\lambda(a_2)=(2,1)$ so that $a=(a_1,a_2)\notin \mathcal [b]_w$, even when $\lambda(b)=\lambda(a)=(4,3,2,1)$. 

We now show that $a$ and $x$ do not strongly operator commute, for any
$x\in [b]_w$. Indeed, if there is $x\in [b]_w$ such that $a$ and $x$ strongly operator commute
then there exists a Jordan frame $\{f_1,\ldots,f_{4}\}$ for $\mathcal V$ such that
$x=\sum_{j=1}^{4} \lambda_j(b) \,f_j$ and $a=\sum_{j=1}^{4} \lambda_j(a) \,f_j$. Since all eigenvalues of $a$ are simple
(have multiplicity one) then we get that $f_1=(e_1,0)$, $f_2=(e_2,0)$, $f_3=(0,e_1)$ and $f_4=(0,e_2)$.
Hence, we now see that $x_1=4\,e_1+3\,e_2=a_1$ and $x_2=2\,e_1+1\,e_2=a_2$ (thus, $x_j$ and $a_j$ strongly operator commute in $\mathcal S$, for $j=1,2$). Notice that this last fact contradicts that $x\in [b]_w$.

Let $F:\mathcal V\rightarrow \mathbb R$ be continuous. Since $[b]_w$ is compact, we conclude that there exists $\bx\in [b]_w$ which is a minimizer
of $F(x-a)$ for $x\in [b]_w$; by the previous arguments, we get that $a$ and $\bx$ do not strongly operator commute. 
\end{example}

The next result shows the converse of Theorem\,\ref{th:strongly} when $\Omega$ is an eigenvalue orbit. 
\begin{proposition}\label{strongglob}
Let $F:\mathcal V\to\mathbb R$ be a Schur-convex spectral function. Let $a,\,b\in \VV$, and let $\bx\in [b]$. Then, the following statements are satisfied:
\begin{enumerate}
\item[(a)] If $a$ and $\bx$ strongly operator commute, then $\bx$ is a global minimizer of $F(x-a)$ for $x\in [b].$
\item[(b)] If $-a$ and $\bx$ strongly operator commute, then $\bx$ is a global maximizer of $F(x-a)$ for $x\in [b].$
\end{enumerate}
\end{proposition}
\begin{proof}
    $(a)$. Suppose that  $a$ and $\bx$ strongly operator commute, then there exists a Jordan frame $\{c_1,\ldots,c_n\}$ such that
    $$a=\sum_{i=1}^n\lambda_i(a)c_i\quad\mbox{and}\quad \bx=\sum_{i=1}^n\lambda_i(\bx)c_i.$$
    Then $\bx-a=\sum_{i=1}^n(\lambda_i(\bx)-\lambda_i(a))c_i$ which implies that for all $x\in[b]$,
    \begin{equation*}
        \lambda(\bx-a)=(\lambda(\bx)-\lambda(a))^\downarrow=(\lambda(x)-\lambda(a))^\downarrow,
    \end{equation*}
    where the second equality is because $\lambda(x)=\lambda(b)=\lambda(\bx)$ since $x$ and $\bx$ remain in the orbit $[b]$. Then, From the Lidskii's inequality \eqref{lidskii}, we get
    \begin{equation}\label{lidskii1}
    \lambda(\bx-a)=(\lambda(x)-\lambda(a))^\downarrow\prec\lambda(x-a),
    \end{equation}
    for all $x\in[b]$. By definition, the fact that $F$ is a Schur-convex spectral function means that its associated symmetric function $f:\mathbb R^n\to\mathbb R$ is Schur-convex. Then, from \eqref{lidskii1} we obtain
    $$F(\bx-a)=f(\lambda(\bx-a))\leq f(\lambda(x-a))=F(x-a),$$
    for all $x\in[b]$. It means that $\bx$ is a global minimizer of $F(x-a)$ for $x\in [b].$

    \noindent $(b).$ The proof is analogous to the previous item. Indeed, by assuming that $-a$ and $\bx$ strongly operator commute we can deduce that for all $x\in[b]$,
    \begin{equation*}
        \lambda(\bx-a)=\lambda(\bx)+\lambda(-a)=\lambda(x)+\lambda(-a).
    \end{equation*}
    Then, the Ky Fan's inequality \eqref{fan} implies that
  $  \lambda(x-a)\prec\lambda(x)+\lambda(-a)=\lambda(\bx-a)$, 
    for all $x\in[b]$. Thus, since $F$ is Schur-convex we conclude that $F(x-a)\leq F(\bx -a)$ for all $x\in[b]$. It means that $\bx$ is a global maximizer of $F(x-a)$ for $x\in [b].$
\end{proof}

Since $[b]_w\subseteq[b]$, the above result is still valid if we substitute $[b]$ by $[b]_w$. This is proved in the next corollary.

\begin{corollary}
    Let $F:\mathcal V\to\mathbb R$ be a Schur-convex spectral function, and let $a,\,b\in \VV$. Let $\bx\in[b]_w$ be such that $a$ (respectively, $-a$) and $\bx$ strongly operator commute. Then, $\bx$ is a global minimizer (respectively, maximizer) of $F(x-a)$ for $x\in [b]_w.$
\end{corollary}
\begin{proof}
   Let $\bx\in [b]_w$ be such that it strongly operator commutes with $a$. Since 
   $[b]_w\subseteq[b]$ we have that $\bx\in[b]$, and by Theorem \ref{strongglob} we deduce that $\bx$ is a global minimizer of $F(x-a)$ for $x\in [b].$ Furthermore, the inclusion $[b]_w\subseteq[b]$ also implies that $\bx$ is a global minimizer of $F(x-a)$ for $x\in [b]_w.$ The proof for the case $\bx$ and $-a$ strongly operator commute is analogous.
\end{proof}

\section{A general commutation principle}\label{general comm}


\begin{theorem}\label{teo local Lidskii in EJAs tutti}
Let $\mathcal V$ be an EJA and let $a\in \mathcal V$. Let  $F:\VV\rightarrow \mathbb R$ be a strictly Schur-convex spectral function, and let $\Omega\subset\VV$ be a weakly spectral set. Let $\bar x$ be a local optimizer (minimizer or maximizer) of $F(x-a)$ for $x\in\Omega$. Then $a$ and  $\bx$ operator commute.
\end{theorem}

The proof of Theorem \ref{teo local Lidskii in EJAs tutti} will be divided into several lemmas. Essentially, we will show that the assumption that $a$ and $\bx$ do not operator commute will lead us to the conclusion that $\bx$ can not be a local optimizer of $F(x-a)$ for $x\in[b]_w$.

Let $a,b\in\VV$. Recall that the property that $a$ and $b$ operator commute is equivalent to the property that $a$ and $b$ are simultaneously (spectral) decomposed; that is, there exists a Jordan frame $\{c_1,\ldots,c_n\}$ such that $a,b\in\VV(c_1,1)\oplus\cdots\oplus\VV(c_n,1)$. If $a$ and $b$ do not operator commute, in some cases, they can still be simultaneously decomposed but in blocks; that is, there exist non-zero idempotents $\{p_1,\ldots,p_k\}$ (for some $2\leq k\leq n$) such that $p_i\circ p_j=0$ if $i\neq j$, $p_1+\ldots+p_k=e$ (which means that $\{p_1,\ldots,p_k\}$ is a complete system of orthogonal idempotents) and such that $a,b\in  \VV(p_1,1)\oplus\cdots\oplus\VV(p_k,1)$. If this is not the case, we say that $a$ and $b$ are totally noncommuting. This concept will be key in our forthcoming lemmas. Below, we give a precise definition.
\begin{definition}
Let $\VV$ be an EJA and let $a,\,b\in\VV$. We say that $a$ and $b$ are \emph{simultaneously block decomposable} if there exists an idempotent $p\notin\{0,e\}$ such that $a,b\in\VV(p,1)\oplus\VV(p,0)$. We say that $a$ and $b$ are \emph{totally (operator) noncommuting } if they are not simultaneously block decomposable.
\end{definition}

We also need to recall some tools from differential geometry. Let $\mathcal N$ and $\MM$ be smooth manifolds embedded into finite dimensional real inner product spaces $\mathcal U$ and $\mathcal W$, respectively. For $u\in\mathcal N$, we denote by $T_u(\mathcal N)$ to the tangent space of $\mathcal N$ at $u$. A smooth map $\Psi:\mathcal N\to\MM$ is a \emph{submersion} at $u\in\mathcal N$ if its differential $D\Psi_u:T_u(\mathcal N)\to T_{\Psi(u)}(\MM)$ is a surjective linear map. 

It is known that $\AutV$ has a smooth manifold (Lie group) structure whose tangent space at the identity (Lie algebra) is $\DerV$. Recently, in \cite[Proposition\,2]{JS2024}, it has been proved that the operator commutativity property can be characterized as follows:
\begin{equation}\label{commute der}
\text{$a$ and $b$ operator commute}\quad\Leftrightarrow\quad \langle Da,b\rangle=0\text{ for all }D\in\DerV.
\end{equation}

In what follows we consider $\mathcal V_t:=\{x\in\mathcal V\, : \ \tr(x)=t\}$, for some $t\in \mathbb R$. Notice that 
$\mathcal V$ is a translation of an hyperplane in $\mathcal V$ and has a natural sub-manifold structure; moreover, we have that $T_{a}(\mathcal V_t)=\mathcal V_0=\{x\in\mathcal V\, : \ \tr(x)=0\}$, for every $a\in \mathcal V_t$.

\begin{lemma}\label{lem local Lidskii1}
Let $\mathcal V$ be an EJA. Fix $a,\,b\in \mathcal V$ and set $t=\tr(a-b)$. Let $\Psi=\Psi_{(a,b)}:\text{Aut}(\mathcal V)\times \text{Aut}(\mathcal V)\rightarrow \mathcal V_t$ be given by $\Psi(X,Y)=Xa-Yb$. Then $\Psi$ is a submersion at $(I,I)$ if and only if $a$ and $b$ are totally noncommuting.
\end{lemma}

\begin{proof}
Observe that $T_{(I,I)}(\rm Aut(\VV)\times\AutV)=\DerV\times\DerV$ and $T_{\Psi(I,I)}(\VV_t)=T_{a-b}(\VV_t)=\VV_0$. Then, the differential map of $\Psi$ at $(I,I)$ is given by
$$D\Psi:\DerV\times\DerV\to\VV_0,\;D\Psi(D,E)=Da-Eb.$$
By definition, $\Psi$ is a submersion at $(I,I)$ if and only if the differential map $D\Psi$ is surjective.
Hence, $\Psi$ is not a submersion if and only if there exists
$f\in \mathcal V_0$, $f\neq 0$ such that $f$ is orthogonal to the range of $D\Psi$.
In turn, this last condition is equivalent to the fact that 
\begin{equation}\label{eq lem submers}
\langle Da, f\rangle=\langle Db, f\rangle=0\quad \text{for every}\quad D\in\text{Der}(\mathcal V)\,,
\end{equation}
which means, because of \eqref{commute der}, that $f$ operator commutes with $a$ and $b$ simultaneously. Let $\gamma_1,\ldots,\gamma_k$ be the distinct eigenvalues of $f$. We know from the first version of the spectral theorem \cite[Theorem III.1.1]{FK} that there exists a unique complete system of orthogonal idempotents $c_1,\ldots,c_k$ such that 
\begin{equation*}\label{unique}
f=\sum_{i=1}^k \gamma_i c_i.
\end{equation*}
Notice that since $f\neq 0$ and $\tr(f)=0$ then $f$ should have at least two different eigenvalues which implies that $c_i\neq e$ for all $i$. Thus, since $f$ operator commutes with $a$ and $b$, we deduce that  $c_1$ also operator commutes with $a$ and $b$. Therefore, $a,b\in\VV(c_1,1)\oplus\VV(c_1,0)$ with $c_1\neq 0$ and $c_1\neq e$. This proves that $a$ and $b$ are simultaneously block decomposable, which means that they are not totally noncommuting. 

Conversely, suppose that $a$ and $b$ are simultaneously block decomposable. Then, there is an idempotent $p\notin\{0,e\}$ such that $a,b\in\VV(p,1)\oplus\VV(p,0)$.  If we let $f=\tr(p)^{-1}\,p-\tr(e-p)^{-1} (e-p)$ then $f$ is a nonzero element in $\VV_0$. Furthermore, relation \eqref{eq lem submers} holds because $\langle Da,p\rangle=\langle Db,p\rangle=0$, since $p$ operator commutes with $a$ and $b$, and $\langle Da,e\rangle=\langle Db,e\rangle=0$, since $e$ operator commutes with each element in $\VV$. Therefore, $f$ is orthogonal to $D\Psi(\DerV\times\DerV)$. The converse now follows from these facts.
\end{proof}

\begin{lemma}\label{pro preparation strong commute}
Let $\mathcal V$ be an EJA, and let $a,\,x\in \mathcal V$.  Suppose that $a$ and $x$ are totally noncommuting. Then, for every $\epsilon >0$, the following statements are satisfied:

\begin{enumerate}
\item [(a)]  there exists $x_\varepsilon\in [x]_w$ such that
$\|x_\varepsilon-x\|\leq \varepsilon$,  
$\lambda(x_\varepsilon-a)\prec \lambda(x-a)$ and $\lambda(x_\varepsilon-a)\neq \lambda(x-a)$;
\item [(b)] there exists $y_\varepsilon\in [x]_w$ such that
$\|y_\varepsilon-x\|\leq \varepsilon$,  
$\lambda(x-a)\prec \lambda(y_\varepsilon-a)$ and $\lambda(y_\varepsilon-a)\neq \lambda(x-a)$.
\end{enumerate}
\end{lemma}
\begin{proof} 
($a$)
Assume that $a$ and $x$ are totally noncommuting. Then, it is clear that $a$ and $x$ do not operator commute. In particular, by item (c) in Proposition \ref{lidskiistrong} we get that 
$(\lambda(x)-\lambda(a))^\downarrow\neq \lambda(x-a)$.

By Lemma \ref{lem local Lidskii1} we see that the map 
$\Psi=\Psi_{(x,a)}:\text{Aut}(\mathcal V)\times \text{Aut}(\mathcal V)\rightarrow \mathcal V_t$ given by $\Psi(X,Y)=Xx-Ya$ is a submersion at $(I,I)$, where $t=\tr(x-a)=\tr(x)-\tr(a)$. Let $\{e_1,\ldots,e_n\}$ be a Jordan frame such that $x-a=\sum_{i=1}^n \lambda_i(x-a)\ e_i$. Consider the continuous curve $\alpha(\cdot):[0,1]\rightarrow (\mathbb R^n)^\downarrow$ given by $\alpha(s)=(1-s)\,\lambda(x-a)+s\,(\lambda(x)-\lambda(a))^\downarrow\in (\mathbb R^n)^\downarrow$, for $s\in [0,1]$.
We now define a continuous curve $$d(\cdot):[0,1]\rightarrow \mathcal V_t \ , \ d(s)=\sum_{i=1}^n \alpha_i(s)\ e_i\ , \ s\in [0,1]\,.
$$
Notice that by construction $d(0)=x-a$ and $\lambda(d(s))=\alpha(s)$, for $s\in [0,1]$. Furthermore, we have that $\alpha(s)\prec \lambda(x-a)$ and $\alpha(s)\neq \lambda(x-a)$, for $s\in (0,1]$. Indeed, the majorization is because $(\lambda(x)-\lambda(a))^\downarrow$ is in the convex hull of the permutations of $\lambda(x-a)$, thanks to the Lidskii's inequality, and $\alpha(s)$ is a convex combination of $\lambda(x-a)$ and $(\lambda(x)-\lambda(a))^\downarrow$. The strict majorization holds because $(\lambda(x)-\lambda(a))^\downarrow\neq \lambda(x-a)$ and $s>0$.   Since $\Psi$ is a submersion at $(I,I)$ then there exist $\eta>0$ such that the restriction 
$\Psi:O_\eta\rightarrow \Psi(O_\eta)$ to the open set $$O_\eta:=\{(X,Y)\in  \text{Aut}(\mathcal V)\times \text{Aut}(\mathcal V) \,:\ \|X-I\|,\,\|Y-I\|\leq \eta\} \ni (I,I)$$ is an open mapping onto the (relative) open set 
$\Psi(O_\eta)\subset \mathcal V_t$ (where, as before, we are considering the product structure on $\text{Aut}(\mathcal V)\times \text{Aut}(\mathcal V)$).
Let $\varepsilon>0$ and assume, without loss of generality, that $\varepsilon <\eta$; then, there exists 
an $s_0=s_0(\varepsilon)\in (0,1)$ such that $d(s)\in \Psi(O_{\hat\varepsilon})$, for $0\leq s\leq s_0$, where $\hat\varepsilon:=\min\{\frac{\varepsilon}{2}, \frac{\varepsilon}{2\|x\|}\}$; notice that we have used that $\Psi(O_{\hat\varepsilon})\subset \mathcal V_t$ is a (relative) open set and that $d(s)$ is a continuous curve such that $d(0)=x-a=\Psi(I,I)$. In particular, there exists 
$(X,Y)\in O_{\hat\varepsilon}$ such that 
$$\Psi(X,Y)=Xx-Ya=d(s_0)=\sum_{i=1}^n \alpha_i(s_0)\ e_i\,.$$
Hence, $Y^{-1}d(s_0)=Y^{-1}Xx-a$ is such that 
$\lambda(Y^{-1}Xx-a)=\lambda(d(s_0))=\alpha(s_0)\prec \lambda(x-a)$, 
$\lambda(Y^{-1}Xx-a)\neq \lambda(x-a)$ and 
$$
\|Y^{-1}Xx-x\|\leq \|Y^{-1}( Xx-x)\|+\|Y^{-1}x-x\|=\|(X-I)x\|+\|(Y-I)x\|\leq \varepsilon\,.
$$Hence, we can take $x_\varepsilon=Y^{-1}Xx\in [x]_w$.

\noindent
($b$) We argue as in the proof of the previous item. Assume that $a$ and $x$ are totally noncommuting so that $a$ and $x$ do not operator commute. In particular, by Ky Fan's inequality \eqref{fan} and Proposition \ref{lidskiistrong} we get that $\lambda(x-a)\prec\lambda(x)+\lambda(-a)$ and
$\lambda(x-a)\neq \lambda(x)+\lambda(-a)$. Then, by Lemma \ref{lem local Lidskii1} we can consider the map
$\Psi=\Psi_{(x,a)}:\text{Aut}(\mathcal V)\times \text{Aut}(\mathcal V)\rightarrow \mathcal V_t$ given by $\Psi(X,Y)=Xx-Ya$ as before, which is a submersion at $(I,I)$. Let $\{e_1,\ldots,e_n\}$ be a Jordan frame such that $x-a=\sum_{i=1}^n \lambda_i(x-a)\ e_i$. Now, we define the continuous curve $\beta(\cdot):[0,1]\rightarrow (\mathbb R^n)^\downarrow$ given by $\beta(s)=(1-s)\,\lambda(x-a)+s\,(\lambda(x)-\lambda(a)^\uparrow)^\downarrow\in (\mathbb R^n)^\downarrow$, for $s\in [0,1]$, and consider the continuous curve 
$$d(\cdot):[0,1]\rightarrow \mathcal V_t \ , \ d(s)=\sum_{i=1}^n \beta_i(s)\ e_i\ , \ s\in [0,1]\,.$$
Notice that by construction $d(0)=x-a$ and $\lambda(d(s))=\beta(s)$, for $s\in [0,1]$. Hence $\lambda(x-a)\prec\beta(s) $ and $\beta(s)\neq \lambda(x-a)$, for $s\in (0,1]$. Indeed, the majorization is because $\lambda(x-a)$ is in the convex hull of the permutations of $\lambda(x)+\lambda(-a)$, thanks to Ky Fan inequality \eqref{fan}, and $\beta(s)$ is a convex combination of $\lambda(x-a)$ and $\lambda(x)+\lambda(-a)$. The strict majorization holds because $\lambda(x)+\lambda(-a)\neq \lambda(x-a)$ and $s>0$. 
Following the same steps as item ($a$), given $\varepsilon >0$ there exist $(X,Y)\in O_{\hat\varepsilon}$ such that $$\Psi(X,Y)=Xx-Ya=d(s_0)=\sum_{i=1}^n \beta_i(s_0)\ e_i\,.$$
Hence, $Y^{-1}d(s_0)=Y^{-1}Xx-a$ is such that 
$\lambda(x-a)\prec\beta(s_0) =\lambda(d(s_0))=\lambda(Y^{-1}Xx-a)$, 
$\lambda(Y^{-1}Xx-a)\neq \lambda(x-a)$ and 
$$
\|Y^{-1}Xx-x\|\leq \|Y^{-1}( Xx-x)\|+\|Y^{-1}x-x\|=\|(X-I)x\|+\|(Y-I)x\|\leq \varepsilon\,.
$$
Hence, we can take $y_\varepsilon=Y^{-1}Xx\in [x]_w$.
\end{proof}

The next result extends Lemma \ref{pro preparation strong commute} to the case where $a$ and $b$ do not operator commute.

 \begin{remark}\label{remark x28}
 Let $\mathcal V$ be an EJA with ${\rm rank}(\mathcal V)\geq 2$ and let $0\neq p\in\mathcal V$ be an idempotent. Recall the Peirce decomposition from \eqref{peirce}, namely
$$
\VV = \VV(p,1)\oplus \VV(p,0) \oplus \VV(p,1/2),
$$
where $\VV(p,1)$ and $\VV(p,0)$ are subalgebras of $\mathcal V$. In fact, $\VV(p,1)$ can be regarded 
as an EJA with unit $p$. In this case, we can embed Aut$(\VV(p,1))$ into Aut$(\VV)$ as follows: given
$X\in$ Aut$(\VV(p,1))$ we extend the action of $X$ by constructing $\tilde X\in$ Aut$(\VV)$ such that
$\tilde X(v_1+v_0+v_{1/2})=X(v_1)+v_0+v_{1/2}$, where $v_i\in \VV(p,i)$, for $i\in \{0,1/2,1\}$.

The previous remarks show that if $x\in\VV$ operator commute with $p$, i.e. $x=\tilde x+ \hat x\in \VV(p,1)\oplus \VV(p,0)$, and $X\in$ Aut$(\VV(p,1))$ then $X(\tilde x)+\hat x=\tilde X x\in [x]_w$. We will need this last fact in the proof of the following result.
\end{remark}

\begin{lemma}\label{pro preparation strong commute2}
Let $\mathcal V$ be an EJA with ${\rm rank}(\mathcal V)\geq 2$, and let $a,\,x\in \mathcal V$ be such that 
they do not operator commute. Then, for every $\varepsilon>0$, the following statements are satisfied:
\begin{enumerate}
\item[(a)] there exists $x_\varepsilon\in [x]_w$ such that
$\|x_\varepsilon-x\|\leq \varepsilon$, 
$\lambda(x_\varepsilon-a)\prec \lambda(x-a)$ and $\lambda(x_\varepsilon-a)\neq \lambda(x-a)$.
\item[(b)]  there exists $y_\varepsilon\in [x]_w$ such that
$\|y_\varepsilon-x\|\leq \varepsilon$,  
$\lambda(x-a)\prec \lambda(y_\varepsilon-a)$ and  $\lambda(y_\varepsilon-a)\neq \lambda(x-a)$.
\end{enumerate}
\end{lemma}

\begin{proof}
($a$)
Consider the class $\mathcal P=\mathcal P_{(a,x)}(\mathcal V)$ consisting of all nonzero idempotents $p\in\VV$ such that 

$$a,x\in\VV(p,1)\oplus\VV(p,0)\quad \mbox{and}\quad a_1\mbox{ and }x_1\mbox{ do not operator commute}, $$
where $a_1$ and $x_1$ are the components (projection) of $a$ and $x$ in $\VV(p,1)$, respectively. Observe that the unit element $e$ is always in $\mathcal P$. Under our present assumptions, we claim that there exists a minimal element 
$q\in \mathcal P$ (that is, an element of $\mathcal P$ with the minimum rank) such that $\text{rank}(q)\geq 2$. We show this claim by induction on $\text{rank}(V)$. 
If $\text{rank}(V)=2$, $e$ is minimal: otherwise, we can deduce that $a$ and $x$ operator commute, against our assumption.

Assume that our claim is true for EJAs with rank at most $n-1\geq 2$ and let $\mathcal V$ be an EJA with $\text{rank}(\mathcal V)=n$. Then, if $e$ is minimal in $\mathcal P$ we are done. Otherwise, there exists
$p\in\mathcal P$ such that $0\neq p \neq e$. Hence, $0\neq e-p\in\mathcal P$ also. If $\text{rank}(p)\geq 2$ then we can consider the EJA $\VV(p,1)$ where $a_1$ and $x_1$ do not operator commute. Since $2\leq \text{rank}(\VV(p,1))<n$, we can apply the induction hypothesis and conclude that there exists a minimal element $q\in \mathcal P_{(a_1,x_1)}(\VV(p,1))$. It follows that $q$ has the required properties.
In case $\text{rank}(p)=1$ then $\text{rank}(e-p)=n-1\geq 2$ and we can replace $p$ by $e-p$ in the previous argument.

Now, let $q$ be a minimal element in $\mathcal P$, and consider the decomposition of $a$ and $x$ in the direct sum $\VV(q,1)\oplus\VV(q,0)$ as
$$a=\tilde a+\hat a,\quad x=\tilde x+\hat x.$$
Since $q$ minimal we have that $\tilde a$ and $\tilde x$ are totally noncommuting in $\VV(q,1)$. Indeed, otherwise, $\tilde a$ and $\tilde x$ are simultaneously block decomposable, which allows us to construct an element in $\mathcal P$ with rank smaller than ${\rm rank}(q)$. Take $\varepsilon>0$; by item ($a$) in Lemma \ref{pro preparation strong commute} there exists
$\tilde x_\varepsilon\in [\tilde x]_w$, where the weak orbit is relative to $\VV(q,1)$, such that 
$\|\tilde x_\varepsilon-\tilde x\|\leq \varepsilon$,
 $\lambda(\tilde x_\varepsilon-\tilde a)\prec\lambda( \tilde x-\tilde a)$ and 
$\lambda(\tilde x_\varepsilon-\tilde a)\neq \lambda(\tilde x-\tilde a)$. 
 Finally, we take $x_\varepsilon=\tilde x_\varepsilon+\hat x\in \mathcal V$; notice that by Remark \ref{remark x28} we get that $x_\varepsilon\in [x]_w$.
On the other hand,
$\|x_\varepsilon-x\|= \|\tilde x_\varepsilon-\tilde x\|\leq \varepsilon$. Finally,
$$\lambda(x_\varepsilon-a)=(\lambda(\tilde x_\varepsilon-\tilde a),\lambda(\hat x-\hat a))^\downarrow
\prec (\lambda(\tilde x-\tilde a),\lambda(\hat x-\hat a))^\downarrow=\lambda(x-a)\,.$$
Arguing as above we also see that 
$\lambda(x_\varepsilon-a)\neq \lambda(x-a)$.

\noindent ($b$) The proof is analogous to the previous item, using item ($b$) in Lemma \ref{pro preparation strong commute}.
\end{proof}

\begin{proof}[Proof of Theorem \ref{teo local Lidskii in EJAs tutti}]
Reasoning by contradiction, suppose that $a$ and $\bx$ do not operator commute. Since $\Omega$ is weakly spectral and $\bx\in\Omega$, we have that $[\bx]_w\subseteq\Omega$. From Lemma\,\ref{pro preparation strong commute2} we deduce that for every $\varepsilon>0$ there exists $x_\varepsilon\in [\bx]_w$ such that
$\|x_\varepsilon-\bx\|\leq \varepsilon$ and such that $x_\varepsilon-a$ is strictly majorized by $\bx-a$. That is, 
\begin{equation}\label{strict}
\lambda(x_\varepsilon-a)\prec \lambda(\bx-a)\quad \mbox{and}\quad \lambda(x_\varepsilon-a)\neq \lambda(\bx-a)\,.
\end{equation}
Let $f:\mathbb R^n\to\RR$ be the symmetric function associated with $F$. We have that $f$ is strictly Schur-convex because $F$ is a strictly Schur-convex spectral function. By applying this fact to the strict majorization \eqref{strict} we get that $F(x_\varepsilon-a)<F(\bx-a)$. It means that $\bx$ is not a local minimizer of $F(x-a)$ for $x\in[\bx]_w$, which implies that $\bx$ is not a local minimizer of $F(x-a)$ for $x\in\Omega$.  We arrived at a contradiction. 

\noindent Arguing as before, we can prove the analogous result for the case of a local maximizer.
\end{proof}

\begin{corollary}\label{coro local Lidskii in EJAs x28}
Let $\mathcal V$ be an EJA and let $a\in \mathcal V$. Let  $F:\VV\rightarrow \mathbb R$ be a strictly Schur-convex spectral function, and let $\Omega\subset\VV$ be a spectral set. Let $\bar x$ be a local optimizer (minimizer or maximizer) of $F(x-a)$ for $x\in\Omega$. Then $a$ and  $\bx$ operator commute.
\end{corollary}

\begin{proof}
Since a spectral set $\Omega\subset \VV$ is a weakly spectral set, we now see that the result follows directly from Theorem 
\ref{teo local Lidskii in EJAs tutti}. 
\end{proof}

\section{Strong commutation principles for local extrema in simple EJAs}\label{strong local}

In this section, we present the following strong commutation principle for local minimizers in automorphism orbits of simple EJAs.  We remark that, unless stated otherwise, all EJAs $\mathcal V$ considered in this section will be simple; in particular, $[b]=[b]_w$, for $b\in\VV$.

\begin{theorem}\label{teo local Lidskii in EJAs}
Let $\mathcal V$ be a simple EJA, let  $F:\mathcal V\rightarrow \mathbb R$ be a strictly Schur-convex spectral function, and let $a,\,b\in \mathcal V$.
\begin{enumerate}
\item[(a)] If $\bar x$ is a local minimizer of $F(x-a)$ for $x\in\mathcal [b]$ then $a$ and $\bx$ strongly operator commute. 
\item[(b)] If $\bar x$ is a local maximizer of $F(x-a)$ for $x\in\mathcal [b]$ then $-a$ and $\bx$ strongly operator commute. 
\end{enumerate}
Moreover, if $\bar x$ is a local minimizer (resp. maximizer) of $F(x-a)$ for $x\in\mathcal [b]$ then it is a global minimizer (resp. maximizer).
\end{theorem}

The proof of Theorem \ref{teo local Lidskii in EJAs} will require the following lemmas.

\begin{lemma}\label{lem:easy}
Let $u=(u_1,u_2),\, v=(v_1,v_2)\in\RR^2$ be such that $u_1+u_2=v_1+v_2$. If we assume that $\Vert u\Vert^2<\Vert v\Vert^2$ then $u$ is strictly majorized by $v$; that is, $u\prec v$ and $u^{\downarrow}\neq v^{\downarrow}$.
\end{lemma}
\begin{proof}
Let $u,v\in\mathbb R^2$ be such that $u_1+u_2=v_1+v_2$. By re-arranging the coordinates of $u$ and $v$, we can assume that $u=u^\downarrow$ and 
$v=v^\downarrow$. If $v_1\leq u_1$ then the hypothesis $u_1+u_2=v_1+v_2$ implies that $v\prec u$. As a consequence we get that $\|u\|^2\geq \|v\|^2$, since $f(x)=x^2$, $x\in \mathbb R$, is a convex function. This last fact contradicts our assumption $\|u\|^2< \|v\|^2$. Hence, we conclude that $u_1<v_1$ (and thus $u_1\neq v_1$); as before, in this case, we get that $u\prec v$. Moreover, $u^\downarrow =(u_1,u_2)\neq (v_1,v_2)=v^\downarrow$.
\end{proof}

\begin{lemma}\label{lem:curve}
Let $\VV$ be a simple Euclidean Jordan algebra of rank $2$. Let $a, x\in \VV$ be such that they operator commute.
\begin{enumerate}
\item[(a)] Suppose that $a$ and $x$ do not strongly operator commute. Then, there exists a continuous curve $\gamma:[0,\pi/2)\to [x]$ such that $\gamma(0)=x$, and $\lambda (\gamma(\theta)-a)\prec \lambda(x-a)$ and  $\lambda (\gamma(\theta)-a)\neq \lambda(x-a)$, for all $\theta\in (0,\pi/2)$.

\item[(b)] Suppose that $-a$ and $x$ do not strongly operator commute. Then, there exists a continuous curve $\gamma:[0,\pi/2)\to [x]$ such that $\gamma(0)=x$, and $\lambda(x-a)\prec \lambda(\gamma(\theta)-a)$ and  $\lambda(x-a)\neq \lambda (\gamma(\theta)-a)$, for all $\theta\in (0,\pi/2)$.
\end{enumerate}
\end{lemma}

\begin{proof} That $a$ and $x$ operator commute means that there exist a Jordan frame $\{e_1,e_2\}$ and scalars $\lambda_1$, $\lambda_2$, $\beta_1$, $\beta_2$ such that
$$a=\lambda_1 e_1 + \lambda_2 e_2\quad\mbox{and}\quad x=\beta_1 e_1+\beta_2 e_2.$$
We first prove item $(a)$. Since $a$ and $x$ do not strongly operator commute we may assume that 
$\lambda_1>\lambda_2$ and $\beta_1<\beta_2$. The fact that $\VV$ is simple ensures that there exists $w\in \mathcal V(e_1,1/2)\cap \mathcal V(e_2,1/2)$ with $\|w\|^2=2$. In this case, we have that $w^2=e_1+e_2$ (see \cite[Proposition IV.1.4.]{FK}). Let $\DD$ be the vector subspace generated by $e_1$, $e_2$ and $w$. That is,
$$\DD:=\RR e_1 +\RR e_2 +\RR w.$$
Thus, $\DD$ is a $3$ dimensional vector space and constitutes a simple Euclidean Jordan algebra of rank $2$ equipped with the inner and Jordan products inherited from $\VV$. Furthermore, $\mathcal D$ is isomorphic to $\mathcal S^2$ (space of $2\times2$ symmetric matrices) in terms of the isomorphism 
$\phi:\mathcal S^2\to \mathcal D$ given by
\begin{equation}\label{eq cons fi}
\phi\begin{pmatrix}
1&0\\
0&0
\end{pmatrix}=e_1,\quad
\phi\begin{pmatrix}
0&0\\
0&1
\end{pmatrix}=e_2,\quad
\phi\begin{pmatrix}
0&1\\
1&0
\end{pmatrix}=w\,.
\end{equation}
For $0\leq \theta< \pi/2$ we set 
\begin{align*}
&e_1(\theta)=\cos(\theta)^2 e_1+\cos(\theta)\,\sin(\theta)w+
\sin(\theta)^2 e_2, \\
&e_2(\theta)=\sin(\theta)^2 e_1-\cos(\theta)\,\sin(\theta)w+
\cos(\theta)^2 e_2 \,.
\end{align*}
Notice that $e_1(\theta),\,e_2(\theta):[0,\pi/2)\rightarrow \mathcal V$
are continuous functions such that 
$\{e_1(\theta),\,e_2(\theta)\}$ is a  Jordan frame for $\mathcal D$, for $\theta\in [0,\pi/2)$. Moreover, notice that $(e_1(0),\,e_2(0))=(e_1,\,e_2)$.
We now construct the continuous curve
\begin{equation}\label{eq cons gam}
\gamma(\theta):=\beta_1 e_1(\theta) + \beta_2 e_2(\theta),\end{equation}
for all $\theta\in [0,\pi/2)$. Observe that $\gamma(0)=x$. Let us prove that item $(a)$ is satisfied. To see this, let $D(\theta)=\phi^{-1}(\gamma(\theta)-a)\in \mathcal S^2$ so that
$\lambda(\gamma(\theta)-a)=\lambda(D(\theta))$,  for $\theta\in [0,\pi/2)$, since $\phi$ is an algebra isomorphism. 
By construction, we have that
$$ D(\theta)
=  V(\theta)
\begin{pmatrix}\beta_1 & 0 \\0 & \beta_2
\end{pmatrix}
V(\theta)^\top-\begin{pmatrix} \lambda_1 & 0 \\0 & \lambda_2
\end{pmatrix}\, \quad\text{with}\quad V(\theta)
=\begin{pmatrix}\cos(\theta) & -\sin(\theta) \\\sin(\theta) & \cos(\theta)
\end{pmatrix} \,,
$$
for all $\theta\in[0,\pi/2)$. Consider
\begin{equation}\label{eq defi R}
R(\theta):=\begin{pmatrix}
\beta_1 & 0 \\
0 & \beta_2
\end{pmatrix}-V(\theta)^\top 
\begin{pmatrix}
\lambda_1 - \lambda_2 & 0 \\
0 & 0
\end{pmatrix}V(\theta)
\end{equation}
Then,
\begin{equation}
\label{defiR}
D(\theta)=V(\theta) \, R(\theta)\, V(\theta)^\top - \lambda_2\, I_2\,.
\end{equation}
Hence, from Lemma\,\ref{lem:easy} we can deduce that $\lambda(R(\theta))\prec \lambda(R(0))$ and $\lambda(R(\theta))\neq \lambda (R(0))$, for  $\theta\in(0,\pi/2)$. Indeed, observe that $\text{tr}(R(\theta))$ is constant for every $\theta\in[0,\pi/2)$, we only need to show that the function  
$[0,\pi/2)\ni \theta\mapsto \Vert R(\theta)\Vert_2 ^2$ is strictly decreasing on $[0,\pi/2)$, where $\|X\|_2=(\lambda_1(X)^2+\lambda_2(X)^2)^{1/2}$ stands for the Frobenius norm of $X\in\mathcal S^2 $. 
%
%
%
 Indeed, since  $\lambda_1-\lambda_2>0$, we have that 
$$ V(\theta)^\top 
\begin{pmatrix}
\lambda_1 - \lambda_2 & 0 \\ 
0 & 0 
\end{pmatrix}
V(\theta) = g(\theta) g(\theta)^\top$$ where $g(\theta)=(\lambda_1-\lambda_2)^{1/2}(\cos(\theta),\sin(\theta))^\top$.
If $M\in \mathcal S^2$ is the diagonal matrix with main diagonal $(\beta_1,\,\beta_2)$ then
$R(\theta)=M-g(\theta) g(\theta)^\top$. Thus, 
$$
\Vert R(\theta)\Vert_2^2=\gamma -2\,\langle M\,g(\theta),\,g(\theta)\rangle
$$
where  $\gamma:=\|g(\theta)\|^4+\beta_1^2+\beta_2^2=(\lambda_1-\lambda_2)^2+\beta_1^2+\beta_2^2\in\mathbb R$ is a constant and 
$$
h(\theta):=\langle M\,g(\theta),\,g(\theta)\rangle =(\lambda_1-\lambda_2)\ (\cos^2(\theta) \, \beta_1+\sin^2(\theta)\,\beta_2),
$$ 
is strictly increasing in  $[0,\pi/2)$.  In fact, 
$$
h'(\theta)= (\lambda_1-\lambda_2) (\beta_2-\beta_1)\, \sin(2\,\theta) >0 \ \ \text{for}\ \ \theta\in (0,\pi/2)\,,
$$
since the function $\sin(2 \theta)$ is positive for $\theta\in (0,\pi/2)$, $\lambda_1-\lambda_2>0$ and $\beta_2-\beta_1>0$. 
Therefore, $\Vert R(\theta)\Vert_2^2 < \Vert R(0)\Vert_2^2$ for every $\theta\in(0,\pi/2)$, and Lemma\,\ref{lem:easy} implies that $\lambda(R(\theta))\prec \lambda(R(0))$ and $\lambda(R(\theta))\neq \lambda (R(0))$ for every $\theta\in(0,\pi/2)$. 
The previous facts and Eq. \eqref{defiR} imply that   
$$ 
\lambda(D(\theta))=\lambda(R(\theta))-\lambda_2\,(1,1)^\top\implies \lambda(D(\theta))\prec \lambda(D(0)) \ , \ \lambda(D(\theta))\neq \lambda(D(0)) \ ,\ \ \theta\in(0,\pi/2)\,.
$$ 
The previous facts imply that $\lambda (\gamma(\theta)-a)\prec \lambda(x-a)$ and $\lambda (\gamma(\theta)-a)\neq \lambda(x-a)$ for every $\theta\in(0,\pi/2)$.

\medskip

\noindent To prove item $(b)$ we notice that since $-a$ and $x$ do not strongly operator commute we may assume that $\lambda_1>\lambda_2$ and $\beta_1>\beta_2$. We now argue as in the first part of the proof and construct
the isomorphism $\phi:\mathcal S^2\to \mathcal D$ according to Eq. \eqref{eq cons fi}, the curve $\gamma(\theta)$ according to \eqref{eq cons gam} and similarly $D(\theta)=\phi^{-1}(\gamma(\theta)-a)$, for $\theta\in [0,\pi/2)$.
Consider $R(\theta)$ as in Eq. \eqref{eq defi R}.
Arguing as before, to check item $(b)$ we only need to show that $[0,\pi/2)\ni \theta\mapsto \Vert R(\theta)\Vert_2 ^2$ is strictly increasing on $[0,\pi/2)$. As before, we only need to check that 
$$
h(\theta):=\langle M\,g(\theta),\,g(\theta)\rangle =(\lambda_1-\lambda_2)\ (\cos^2(\theta) \, \beta_1+\sin^2(\theta)\,\beta_2),
$$ 
is strictly decreasing in  $[0,\pi/2)$; but this fact follows from the inequality
$$
h'(\theta)= (\lambda_1-\lambda_2) (\beta_2-\beta_1)\, \sin(2\,\theta) <0 \ \ \text{for}\ \ \theta\in (0,\pi/2)\,,
$$
where we used that the function $\sin(2 \theta)$ is positive for $\theta\in (0,\pi/2)$, $\lambda_1-\lambda_2>0$ and $\beta_2-\beta_1<0$ in the present case. 

\medskip
\end{proof}


\begin{proof}[Proof of Theorem \ref{teo local Lidskii in EJAs}]
$(a)$. Observe that $[b]=[b]_w$ because the algebra is simple. Since $\bx$ is a local minimizer of $F(x-a)$ for $x\in\mathcal [b]$, from the commutation principle given in Theorem\,\ref{teo local Lidskii in EJAs tutti} we deduce that $a$ and $\bx$ operator commute. 
 Then, there exists a Jordan frame
$\{e_1,\ldots,e_n\}$ and real numbers $\beta_1,\ldots,\beta_n\in \mathbb R$ such that 
\begin{equation}\label{eq ab comm1}
a=\sum_{i=1}^n \lambda_i(a)\,e_i \quad \text{and}\quad \bar x=\sum_{i=1}^n \beta_i \,e_i\,.
\end{equation}
Reasoning by contradiction, suppose that $a$ and $\bx$ do not strongly operator commute. Then, there exist indices $1\leq j < k\leq n$ such that
\begin{equation}\label{assump}
\lambda_j(a)>\lambda_k(a)\quad\mbox{and}\quad\beta_j<\beta_k.
\end{equation}
Below we show that this last fact implies that $\bar x$ is not a local minimizer of $F(x-a)$ for $x\in[b]$.

Recall that every idempotent $p\in\mathcal V$ induces the Peirce decomposition $\mathcal V=\mathcal V(p,1)\oplus \mathcal V(p,0)\oplus \mathcal V(p,1/2)$, see \eqref{peirce}. Let $p:=e_j+e_k$, which turns out to be an idempotent in $\VV$. Observe that $\VV(p,1)$ is a simple Euclidean Jordan algebra of rank $2$, and 
\begin{equation}\label{eq defi hata hatx}
\hat a:=\lambda_j(a)e_j+\lambda_k(a) e_k\quad\mbox{and}\quad\hat x:=\beta_j e_j+\beta_k e_k,
\end{equation}
are elements of $\VV(p,1)$. Then, $\hat a$ and $\hat x$ operator commute in $\VV(p,1)$ but not strongly because of \eqref{assump}. From Lemma\,\ref{lem:curve}, item $(a)$, we deduce that there exists a continuous curve $\gamma:[0,\pi/2)\to [\hat x]$ such that $\gamma(0)=\hat x$, and 
\begin{equation}\label{key2}
\lambda (\gamma(\theta)-\hat a)\prec \lambda(\hat x-\hat a)\quad\mbox{and}\quad \lambda (\gamma(\theta)-\hat a)\neq \lambda(\hat x-\hat a),
\end{equation}
for all $\theta\in (0,\pi/2)$, where $\lambda(\cdot)$ is taken over $\VV(p,1)$. We now construct the continuous curve
$$
\bar x(\theta):=\sum_{i=1,\,i\notin\{j,k\}}^n \beta_i\ e_i+\gamma(\theta)\in \VV\ , \ \ \text{for} \ \theta\in [0,\pi/2)\,.
$$
By construction, $\lambda(\bar x(\theta))=\lambda(b)$, and hence $\bar x(\theta)\in[b]$, for $\theta\in [0,\pi/2)$.
Also, notice that $\bar x(0)=\bar x$. We claim that
\begin{equation}\label{claim_th}
\lambda(\bar x(\theta)-a)\prec \lambda(\bar x-a) \quad \text{and} \quad \lambda(\bar x(\theta)-a)\neq \lambda(\bar x-a) \ , \ \ \text{for} \
\theta\in (0,\pi/2)\,.
\end{equation}
Indeed, we have that
$$\bx(\theta)-a=\overbrace{\sum_{i=1,\,i\notin\{j,k\}}^n (\beta_i-\lambda_i(a))\ e_i}^{\in\VV(p,0)}+\overbrace{(\gamma(\theta)-\hat a)}^{\in\VV(p,1)}.
$$
Then, 
$$P_1\lambda(\bx(\theta)-a)=\left(\left(\beta_i-\lambda_i(a)\right)_{i=1,\,i\notin\{j,k\}}^n,\,\lambda(\gamma(\theta)-\hat a)\right),$$
for some permutation matrix $P_1$. Analogously, 
$$P_2\lambda(\bx-a)=\left((\beta_i-\lambda_i(a))_{i=1,\,i\notin\{j,k\}}^n,\,\lambda\left(\hat x-\hat a\right)\right),$$
for some permutation matrix $P_2$. Now, we use the property that if $\gamma\prec \eta$ and $\delta\prec\nu$ then $(\gamma,\delta)\prec (\eta,\nu)$, and if any of the first two majorization relations is strict, then the majorization relation $(\gamma,\eta)\prec (\eta,\nu)$ is also strict. Indeed, since the first $n-2$ terms of $P_1\lambda(\bx(\theta)-a)$ and $P_2\lambda(\bx-a)$ are equal, and because of the strict majorization \eqref{key2}, we conclude that $P_1\lambda(\bx(\theta)-a)$ is strictly majorized by $P_2\lambda(\bx-a)$ and the claim \eqref{claim_th} follows from these facts.

Finally, given $\varepsilon>0$ we choose $\theta\in (0,\pi/2)$ such that 
$\|\bar x-\bar x(\theta)\|\leq \varepsilon$. Furthermore, as $F$ is a strictly Schur-convex spectral function, its associated symmetric function $f$ is also strictly Schur-convex. Since $\lambda(\bar x(\theta)-a)$ is strictly majorized by $\lambda(\bar x-a)$, see \eqref{claim_th},  we deduce that 
$$F(\bar x(\theta)-a)<F(\bar x-a).$$
This last fact contradicts our assumption that $\bar x$ is a local minimizer of $F(x-a)$ for $x\in[b]$. Hence $a$ and $\bar x$ strong operator commute. Furthermore, from Theorem\,\ref{strongglob}, we conclude that $\bx$ is a global minimizer of $F(x-a)$ for $x\in\mathcal [b]$.

\noindent $(b)$. The proof follows a similar argument to that in the proof of the previous item. 
Since $\bx$ is a local maximizer of $F(x-a)$ for $x\in\mathcal [b]$ we deduce that $a$ and $\bx$ operator commute. Then, there exist a Jordan frame
$\{e_1,\ldots,e_n\}$ and real numbers $\beta_1,\ldots,\beta_n\in \mathbb R$ such that  \eqref{eq ab comm1} holds.
If we assume that $a$ and $\bx$ do not strongly operator commute then there exist indexes $1\leq j < k\leq n$ such that
\begin{equation}\label{assump2}
\lambda_j(a)>\lambda_k(a)\quad\mbox{and}\quad\beta_j>\beta_k.
\end{equation}
we now define $\hat a,\,\hat x\in\mathcal V(p,1)$ as in \eqref{eq defi hata hatx}, where $p=e_j+e_k$.
From Eq. \eqref{assump2} we see that $-\hat a$ and $\hat x$ do not strongly operator commute. Hence, from Lemma\,\ref{lem:curve}, item $(b)$, there exists a continuous curve $\gamma:[0,\pi/2)\to [\hat x]$ such that $\gamma(0)=\hat x$, and 
$$
 \lambda(\hat x-\hat a)\prec \lambda (\gamma(\theta)-\hat a)\quad\mbox{and} \quad \lambda(\hat x-\hat a)\neq  \lambda (\gamma(\theta)-\hat a),
$$
for all $\theta\in (0,\pi/2)$, where $\lambda(\cdot)$ is taken over $\VV(p,1)$. We now construct the continuous curve
$$
\bar x(\theta):=\sum_{i=1,\,i\notin\{j,k\}}^n \beta_i\ e_i+\gamma(\theta)\in \VV\ , \ \ \text{for} \ \theta\in [0,\pi/2)\,.
$$ Arguing as in the first part of the proof we get that $\lambda(\bar x-a)\prec \lambda(\bar x(\theta)-a)$ and 
$\lambda(\bar x-a)\neq \lambda(\bar x(\theta)-a)$, for $\theta\in (0,\pi/2)$. 
As before, the previous facts imply that $\bx$ is not a local maximizer of $F(x-a)$ for $x\in [b]$, against our assumption.
\end{proof}

The following remark shows that Theorem\,\ref{teo local Lidskii in EJAs} can not be extended for nonsimple EJAs.

\begin{remark}\label{rem no ejem2}
Let $\VV$ be a possibly not simple EJA, and let $b\in \VV$. Then, there exists a finite family $\{b_j \}_{j=1}^m$ of elements in $\VV$
such that: $b_1=b$,
\begin{enumerate}
\item\label{it1} $\lambda(b_j)=\lambda(b)$, for $1\leq j\leq m$;
\item\label{it2} For $1\leq \ell\neq j\leq m$, $[b_\ell]_w\neq [b_j]_w$ and $ [b]=\bigcup _{j=1}^m [b_j]_w\,.$
\item\label{it3} For $1\leq \ell\neq j\leq m$ we have that 
$\text{dist}([b_j]_w,\,[b_\ell]_w)=\inf\left \{ || c-d|| \ : \ c\in [b_j]_w \ , \ d\in [b_\ell]_w\right \} \ > \ 0\,.$
		This last claim follows from the fact that $[b_\ell]_w$ and $[b_j]_w$ are compact sets with empty intersection.
\end{enumerate}
Consider now the notation of Example \ref{no exa1}. Thus, we let $\mathcal V=\mathcal S\times \mathcal S$, where
$\mathcal S$ is a simple EJA of rank two. Recall that in this context we can construct $a,\,b \in\mathcal V$ such that $x$ and $a$ do not strongly operator commute, for any $x\in [b]_w$. Moreover, notice that in that example, $a\in [b]$, since $\lambda(a)=\lambda(b)$. 
Given a strictly convex spectral function $F:\mathcal V\rightarrow \mathbb R$, let $\bx\in [b]_w$ be a global (and hence a local) minimizer of $F(x-a)$ for $x\in [b]_w$. We now claim that $\bx\in [b]$ is a local minimizer of $F(x-a)$ for $x\in [b]$ (i.e., is a local minimizer for the larger domain $[b]$). 
In fact, by item \ref{it3}. above we see that  $[b]_w=[b_1]_w$ is separated from $$\bigcup _{2\leq j\leq m} \  [b_j]_w\,.$$ That is, we can find a sufficiently small neighborhood of $\bx$ in $[b]$ that is actually contained in $[b_1]_w$. This last fact implies the previous claim. Of course, $\bx$ is not a global minimizer of $F$ in $[b]$, since $a\in [b]$ and $\lambda(a-a)=0_n\prec \lambda(x-b)$ and $0_n\neq \lambda(x-b)$, for every $x\in [b]_w$. Hence, $F(0)=F(a-a)< F(x-a)$, for every $x\in [b]_w$ (these facts follow from the majorization relation $0_n\prec \alpha$, for $\alpha \in \mathbb R^n$ such that $\sum_{i=1}^n \alpha_i=0$). Finally, notice that by construction $a$ and $\bx\in [b]_w$ do not strongly operator commute. 
\end{remark}

\begin{remark}\label{general}
For simplicity of presentation, Theorems  \ref{th:strongly}, \ref{teo local Lidskii in EJAs tutti}, and \ref{teo local Lidskii in EJAs} require $F(\cdot)$ to be strictly Schur-convex on the entire space $\VV$. However, those theorems are still valid if $F(\cdot)$ is strictly Schur-convex in some spectral set $K\subseteq\VV$ such that $\Omega-a\subseteq K$ (with $\Omega=[b]$ for Theorem \ref{teo local Lidskii in EJAs}). 
\end{remark}

\section{An application to the problem of minimizing the condition number in EJA}\label{application}

In this section, we assume that $\VV$ is an EJA of rank $n\geq 2$. Recall that $\lambda_i(x)$ are the eigenvalues of $x\in\VV$ arranged in nonincreasing order: $\lambda_1(x)\geq\cdots\geq\lambda_n(x)$. Let $\VV_+$ denote the interior of the cone of square elements of $\VV$ (symmetric cone of $\VV$). It can be characterized as $\VV_+=\{x\in\VV:\lambda_n(x)>0\}$.  

The condition number of $x\in \VV_+$ is defined as
$$\mathfrak{c}(x)=\frac{\lambda_1(x)}{\lambda_n(x)}.$$
Given $a\in \VV_+$, there is interest in studying the problem of perturbing $a$ in such a way as to diminish its condition number as much as possible, see \cite{Seeger2022, Greif}. This problem can be formulated
\begin{equation}\label{condition_number}
\min_{x\in\Omega} \mathfrak{c}(x+a),
\end{equation}
where $\Omega$ is some set of admissible perturbations. It is known that $\mathfrak{c}(\cdot)$ is Schur-convex on $\VV_{+}$, but not necessarily strictly Schur-convex. Thus, we cannot apply our results directly to study this problem since we require strict Schur-convexity.

To overcome the above drawback, we consider a vector containing not only the condition number but also the ratio between the $i$th eigenvalue and the $(n-i+1)$th eigenvalue, for $i=1,\ldots,\lfloor n/2\rfloor$. 
 This is specified in the next definition.
\begin{definition}
Let $x\in\VV_{+}$, and let
$$\kappa_i(x):=\frac{\lambda_i(x)}{\lambda_{n-i+1}(x)}\quad\text{for }i=1,\ldots,\lfloor n/2\rfloor.$$
The \emph{condition vector} of $x$ is defined as the vector $\kappa(x)=\left(\kappa_1(x),\ldots,\kappa_{\lfloor n/2\rfloor}(x)\right)\in\RR^{\lfloor n/2\rfloor}$.
\end{definition}
It is straightforward to see that the entries of $\kappa(x)$ are arranged in nonincreasing order and that 
$\kappa_1(x)=\mathfrak{c}(x)$. Because of those facts, it can be proved that
$$\lfloor n/2\rfloor^{-1/2}\|\kappa(x)\|\leq \mathfrak{c}(x)\leq \|\kappa(x)\|,$$
for every $x\in\VV_+$, where $\|\kappa(x)\|$ denotes the euclidean norm of $\kappa(x)$. This shows that the behavior of $\|\kappa(x)\|$ is comparable to the behavior of $\mathfrak{c}(x)$. That is, if $\|\kappa(x)\|$ is small or large, then $\mathfrak{c}(x)$ is small or large, respectively. Thus, we can analyze the problem \eqref{condition_number} by solving the problem
\begin{equation}\label{cond_vector}
\min_{x\in\Omega} \Vert\kappa(x+a)\Vert.
\end{equation}
Our goal is to describe some fundamental properties of the solutions of \eqref{cond_vector} by relying on the tools we developed in the previous sections. For $x\in\VV_+$, let us denote $\kappa_{\|\cdot\|}(x):=\Vert\kappa(x)\Vert$. The next proposition shows that $\kappa_{\|\cdot\|}$ is a strictly Schur-convex spectral function on $\VV_+$. Let $\RR^n_+$ denote the interior of the positive orthant of $\RR^n$. 
\begin{proposition}
The following statements are satisfied.
\begin{enumerate}
\item[(a)] $\kappa_{\|\cdot\|}:\VV_+\to\RR$ is a spectral function with associated symmetric function $\phi_{\|\cdot\|}:\RR^n_+\to\RR$ given by $\phi_{\|\cdot\|}(u)=\Vert\phi(u)\Vert$, with
$$\phi(u)=\left(\frac{u^{\downarrow}_1}{u^{\downarrow}_{n}}, \frac{u^{\downarrow}_2}{u^{\downarrow}_{n-1}},\ldots, \frac{u^{\downarrow}_{\lfloor n/2\rfloor}}{u^{\downarrow}_{n-\lfloor n/2\rfloor+1}}\right)\in\RR^{\lfloor n/2\rfloor}.$$
\item[(b)] $\kappa_{\|\cdot\|}$ is a strictly Schur-convex spectral function on $\VV_+$. 
\end{enumerate}
\end{proposition}
\begin{proof}
(a) It is clear from the definition of $\kappa_{\|\cdot\|}$.\\
(b) We first show that $\phi_{\|\cdot\|}$ is Schur-convex on $\RR^n_+$. Indeed, let $u,\,v\in \RR^n_+$ be such that $u\prec v$. Without loss of generality, we assume that $u=u^\downarrow$ and $v=v^\downarrow$. Hence, by the relations between majorization and convex functions (see \cite[Theorem\,5.A.1]{Marshall}), we get that
$$(\RR^n_+)^\downarrow\ni(u_n^{-1},\ldots, u_1^{-1})\prec_w (v_n^{-1},\ldots, v_1^{-1})\in (\RR^n_+)^\downarrow\,.$$
As a consequence, we get the following relations for the entry-wise product vectors (see \cite[Theorem\,5.A.4.f]{Marshall}) $(\mathbb R_+^n)^\downarrow\ni (u_i/ u_{n-i+1})_{i=1}^n \prec_w (v_i/ v_{n-i+1})_{i=1}^n\in (\mathbb R_+^n)^\downarrow \,.$ In particular,
\begin{equation}\label{mitad}
(\mathbb R_+^{\lfloor n/2\rfloor})^\downarrow\ni(u_i/ u_{n-i+1})_{i=1}^{\lfloor n/2\rfloor} \prec_w (v_i/ v_{n-i+1})_{i=1}^{\lfloor n/2\rfloor}\in (\mathbb R_+^{\lfloor n/2\rfloor})^\downarrow,
\end{equation}which implies that 
$\phi_{\|\cdot\|}(u)\leq \phi_{\|\cdot\|}(v)$.\\
 We now show that $\phi_{\|\cdot\|}$ is strictly Schur-convex on $\RR^n_+$. As before,
we consider $u,v\in \RR^n_+$ such that $u\prec v$, and assume that $u=u^\downarrow$ and $v=v^\downarrow$. 
 Notice that, in particular, these vectors satisfy the normalization condition  $\tr(u):=u_1+\cdots+u_n=v_1+\cdots+v_n:=\tr(v)>0$.
  Suppose further that $\phi_{\|\cdot\|}(u)=\phi_{\|\cdot\|}(v)$.  Notice that $\|\cdot\|$ is a strictly Schur-convex function, and it is a strictly increasing function on $\RR^n_+$ (for the componentwise ordering in $\RR^n_+$). Hence, from \cite[Theorem\,3.A.8]{Marshall} and the submajorization \eqref{mitad} we have that 
 $u_i/ u_{n-i+1}=v_i/ v_{n-i+1}$ for all $i=1,\ldots,\lfloor n/2\rfloor$. In particular, we get that $u_1/u_n=v_1/v_n$; now, using that $u\prec v$ implies that $v_n\leq u_n$ and $v_1\geq u_1$, we see that $u_i=v_i$ for $i\in\{1,n\}$. We can continue in this way and show that $u_i=v_i$, for $i\in \{j,\,n-j+1\}$ and $1\leq j\leq \lfloor n/2\rfloor$. Thus, we get $u=v$ when $n$ is even. When $n$ is odd, the previous fact together with the normalization condition  $\tr(u)=\tr(v)$  imply that $u=v$.
\end{proof}

Finally, the following theorem characterizes the solutions of \eqref{cond_vector} and, in some cases, reduces \eqref{cond_vector} to a minimization problem formulated in $\RR^n$. The proof of each item follows by a direct application of Theorems \ref{teo local Lidskii in EJAs tutti}, \ref{th:strongly} and \ref{teo local Lidskii in EJAs}, respectively (see also Remark\,\ref{general}).
\begin{theorem}
Let $a\in\VV_+$  and let $\Omega\subseteq\VV$ be such that $\Omega+a\subseteq\VV_+$. Consider the problem \eqref{cond_vector}.
Then, the following statements are satisfied:
\begin{enumerate}
\item[(a)] Suppose that $\Omega$ is a weakly spectral set. If $\bx$ is a solution of \eqref{cond_vector}, then $a$ and $\bx$ operator commute.
\item[(b)] Suppose that $\Omega$ is a spectral set with $Q\subseteq\RR^n$ as its associated symmetric set. If $\bx$ is a solution of \eqref{cond_vector}, then $-a$ and $\bx$ strongly operator commute. Furthermore,
$$\min_{x\in\Omega} \Vert\kappa(x+a)\Vert=\min_{u\in Q} \Vert\phi(u-\lambda(-a))\Vert.$$
\item[(c)] Let $b\in\VV$ be such that $[b]+a\subset\VV_+$. Suppose that $\VV$ is simple. If $\bx$ is a local solution of
\begin{equation}\label{cond_vector3}
\min_{x\in[b]} \Vert\kappa(x+a)\Vert,
\end{equation}
then $-a$ and $\bx$ strongly operator commute. Furthermore, $\bx$ is a global solution of \eqref{cond_vector3}, and
$$\min_{x\in[b]} \Vert\kappa(x+a)\Vert= \Vert\phi(\lambda(b)-\lambda(-a))\Vert.$$
\end{enumerate}
\end{theorem}

\end{document}